\documentclass{amsart}
 \pdfoutput=1
\usepackage{amsmath,amssymb,amsthm,verbatim,paralist,graphicx}
\usepackage[active]{srcltx}
\usepackage[T1]{fontenc}
\usepackage{yfonts}
\usepackage{color}
\DeclareMathOperator{\rk}{r}
\DeclareMathOperator{\cl}{cl}

\newtheorem{theorem}{Theorem}
\newtheorem{prop}{Proposition}
\newtheorem{lemma}{Lemma}
\newtheorem{corollary}{Corollary}
\newtheorem{definition}{Definition}
\newtheorem{conjecture}{Conjecture}

\begin{document}
\title{Sticky matroids and Kantor's Conjecture} \author{Winfried
  Hochst\"attler\and Michael Wilhelmi}
\address{FernUniversit\"at in Hagen}
\email{\{Winfried.Hochstaettler, Michael.Wilhelmi\}@FernUniversitaet-Hagen.de}

\begin{abstract}
  We prove the equivalence of Kantor's Conjecture and the Sticky
  Matroid Conjecture due to Poljak und Turz\'ik.
\end{abstract}

\dedicatory{
  Dedicated to Achim Bachem on the occasion of his 70th birthday.}

\maketitle

\section{Introduction}
The purpose of this paper is to prove the equivalence of two classical
conjectures from combinatorial geometry. Kantor's Conjecture
\cite{Kantor} adresses the problem whether a combinatorial geometry
can be embedded into a modular geometry, i.e.,\ a direct product of
projective spaces. He conjectured that for finite geometries this is
always possible if all pairs of hyperplanes are modular.

The other conjecture, the Sticky Matroid Conjecture (SMC) due to
Poljak and Turz\'ik \cite{PoljakTurzik} concerns the question whether
it is possible to glue two matroids together along a common part. They
conjecture that a ``common part'' for which this is always possible, a
sticky matroid, must be modular.  It is well-known (see eg.\
\cite{Oxley}) that modular matroids are sticky and easy to see
\cite{PoljakTurzik} that modularity is necessary for ranks up to
three. Bachem and Kern~\cite{BachemKernSticky} proved that a rank-4
matroid that has two hyperplanes intersecting in a point is not
sticky. They also stated that a matroid is not sticky if for each of
its non-modular pairs there exists an extension decreasing its modular
defect. The proof of this statement had a flaw which was fixed by
Bonin~\cite{BoninSticky}.  Using a result of Wille~\cite{Wille67} and
Kantor~\cite{Kantor} this implies that the sticky matroid conjecture
is true if and only if it holds in the rank-4 case. 
Bonin~\cite{BoninSticky} also showed that a matroid of rank $\rk \ge
3$ with two disjoint hyperplanes is not sticky and that non-stickiness
is also implied by the existence of a hyperplane and a line that do
not intersect but can be made modular in an extension.

We generalize Bonin's result and show that a matroid is not sticky if
it has a non-modular pair  that admits an extension decreasing its
modular defect. Moreover by showing the existence of the proper
amalgam of two arbitrary extensions of the matroid we prove that in the rank-4 case this condition is also
necessary for a matroid not to be sticky. 
As a consequence from every
counterexample to Kantor's conjecture arises a matroid  that can be extended in finite steps
to a
counterexample of  the (SMC), implying the equivalence of the two conjectures.
A further consequence of our results is the equivalence of both
conjectures to the following:

\begin{conjecture}\label{conj:1}
  In every finite non-modular matroid there exists a non-modular pair and a
  single-element extension decreasing its modular defect.
\end{conjecture}

Finally, we present an example proving that the (SMC), like Kantor's
Conjecture fails in the infinite case. 

We assume familiarity with
matroid theory. The standard reference is \cite{Oxley}.

\section{Our results}

Let $ M $ be a matroid with groundset $ E $ and rank function $ \rk $.
We define the {\em modular defect $\delta(X,Y) $} of a
pair of subsets $ X,Y \subseteq E $ as 
\[\delta(X,Y)= \rk(X) + \rk(Y) - \rk(X \cup Y)
- \rk(X \cap Y).\] By submodularity of the rank function, the modular
defect is always non-negative. If it equals zero, we call $ (X,Y) $ a
{\em modular pair}.  A matroid is called {\em modular} if all pairs
of flats form a modular pair.

An {\em extension} of a matroid $M$ on a groundset $E$ is a matroid $N$ on a
groundset $\mbox{$F\supseteq E$}$ such that $M = N|E$. If $N_1, N_2$ are extensions of
a common matroid $M$ with groundsets $F_1, F_2, E$ resp.\ such that
$F_1 \cap F_2=E,$ then a matroid $A(N_1,N_2)$ with groundset $F_1\cup F_2$ is
called an {\em amalgam} of $N_1$ and $N_2$ if $\mbox{$A(N_1,N_2)|F_i=N_i$}$ for $i=1,2$.

\begin{theorem}[Ingleton see \cite{Oxley} 11.4.10 (ii)]\label{theo:modMatrSticky} If $M$ is a
  modular matroid then  for any pair $(N_1,N_2)$ of extensions of
  $M$ an amalgam exists.
\end{theorem}

We found a proof of this result only for finite matroids (see eg.\
\cite{Oxley}). We will show that it also holds for infinite matroids
of finite rank.

\begin{conjecture}[Sticky Matroid Conjecture (SMC) \cite{PoljakTurzik}]
  If $M$ is a matroid such that for all pairs $(N_1,N_2)$ of extensions
  of $M$ an amalgam exists, then $M$ is modular.
\end{conjecture}

The following preliminary results concerning the (SMC) are known:

\begin{theorem}[\cite{PoljakTurzik, BachemKernSticky, BoninSticky}]\label{theo:knownForSMC} Let $M$ be a matroid.
  \begin{enumerate}
  \item If $\rk(M)\le 3$ then the (SMC)
    holds for $M$.
  \item If the $(SMC)$ holds for all rank-4 matroids, then it is true
    in all ranks.
  \item Let $l$ be a line and $H$ a hyperplane in $M$ such that
    $\mbox{$l \cap H=\emptyset$}$. If $M$ has an extension $M'$ such
    that $\rk_{M'}\left(\cl_{M'}(l)\cap\cl_{M'}(H)\right) =1$, then
    $M$ is not sticky.
\item If $M$ has two disjoint hyperplanes then is not sticky.
  \end{enumerate}
\end{theorem}

We will generalize the last two assertions and prove:

\begin{theorem}\label{theo:extbonin}
  Let $M$ be a matroid, $X$ and $Y$ two flats such that
  $\delta(X,Y)>0$.  If $M$ has an extension $M'$ such that
$\delta_{M'}\left(\cl_{M'}(X),\cl_{M'}(Y)\right) < \delta(X,Y)$ then $M$ is not
    sticky.
\end{theorem}
We postpone the proof of Theorem~\ref{theo:extbonin} to
Section~\ref{sec:extbonin}.

We call a matroid {\em hypermodular} if
each pair of hyperplanes forms a modular pair.
With this notion we can rephrase Kantor's Conjecture.

\begin{conjecture}[Kantor \cite{Kantor}, page 192]
Every finite hypermodular matroid embeds into a
  modular matroid.
\end{conjecture}


Like the (SMC) Kantor's Conjecture can be reduced to the rank-4
case (see Corollary~\ref{corollar:KantorRang4}, Section~\ref{sec:embed}).

  Next, we consider the
correspondence between single-element extensions of matroids and modular
cuts.

\begin{definition}\label{def:modularerFilter}
A set $ \mathcal{M} $ of flats of a matroid $ M $ is called a {\em modular cut of 
$ M $} if the following holds:\enlargethispage{1cm}
\begin{enumerate}
\item If $ F \in \mathcal{M} $ and $ F' $ is a flat in $ M $ with $ F' \supseteq F $, then $ F' \in \mathcal{M} $. 
\item If $ F_1, F_2 \in \mathcal{M} $ and $ (F_1,F_2) $ is a modular pair, then $ F_1 \cap F_2 \in \mathcal{M} $. 
\end{enumerate}
\end{definition}

\begin{theorem}[Crapo 1965 \cite{Crapo}]
  There is a one-{to}-one-correspondence between the single-element extensions 
  $M+_{\mathcal{M}}p$ of a matroid $M$ and the modular cuts $ \mathcal{M} $ of
  $M$. $\mathcal{M}$ consists precisely of the set of flats of $ M $ containing the new point $ p $ in $M+_{\mathcal{M}}p$. 

\end{theorem}

The set of all flats of a matroid $ M $ is a modular cut, the {\em
  trivial modular cut}, corresponding to an extension with a loop.
The empty set is a modular cut corresponding to an extension with a
coloop, the only single-element extension increasing the rank of $ M $.
For a flat $ F $ of $ M $, the set $ \mathcal{M}_F = \{ G \mid G \text{ is a
flat of } M \text{ and } G \supseteq F \} $ is a modular cut of $ M $. We
call it a {\em principal modular cut}. We say that in the
corresponding extension the new point is {\em freely added} to $ F $.
A {\em modular cut $\mathcal{M}_{\mathcal{A}}$ generated by a set of flats $
  \mathcal{A} $} is the smallest modular cut containing $
\mathcal{A} $.  

The following is immediate from Theorem 7.2.3 of \cite{Oxley}.
\begin{prop}\label{intersectablePairs}
  If $(X,Y)$ is a non-modular pair of flats of a matroid $M$, then
  there exists an extension decreasing its modular defect (we call the pair { \em intersectable}) if and only
  if the modular cut generated by $X$ and $Y$ is not the principal
  modular cut $\mathcal{M}_{X \cap Y}$.
\end{prop}

We call a matroid {\em OTE (only trivially extendable)} if all of its
modular cuts different from the empty modular cut are principal.

Most of this paper will be devoted to the proof of the following theorem.

\begin{theorem}\label{theo:OTErank4sticky}
  If $M$ is a rank-4 matroid that is OTE, then $M$ is sticky.
\end{theorem}

As we will prove with Theorem~\ref{theo:nine},
Theorem~\ref{theo:extbonin} implies that a matroid that is not OTE is
not sticky. Hence Theorem~\ref{theo:OTErank4sticky} implies that for
rank-4 matroids being sticky is equivalent to being OTE. Since the
(SMC) is reducible to the rank-4 case, it is equivalent to the
conjecture that every $\mbox{rank-$4$}$ matroid that is OTE is already
modular.  For finite matroids, this is our Conjecture~\ref{conj:1},
which is also reducible to the rank-4 case (see the remark after the
proof of Corollary \ref{corollar:KantorRang4}).

  Like Kantor's Conjecture our Conjecture~\ref{conj:1} is no
  longer true in the infinite case. This will be a consequence of the
  following theorem, proven in Section~\ref{sec:embed}.

\begin{theorem}
  Every finite matroid can be extended to a (not
  necessarily finite) matroid of the same rank that is OTE.
\end{theorem}

Starting from, say, the V\'amos matroid this yields an infinite rank-4
non-modular matroid that is OTE, hence a counterexample to the (SMC)
in the infinite case.

Finally, Theorem~\ref{satz:ErsterEinbettungssatz} will imply that any
finite counterexample to Kantor's Conjecture can be embedded into a
finite non-modular matroid that is OTE. In the rank-$4$ case any
counterexample to Kantor's Conjecture this way yields a finite
counterexample to the (SMC). We wil show in Corollary
\ref{corollar:KantorRang4} that Kantor's Conjecture is reducible to
the rank-4 case, hence the (SMC) implies Kantor's Conjecture. It had
already been observed by Faigle (see \cite{BachemKernSticky}) and was
explicitely mentioned by Bonin in \cite{BoninSticky} that Kantor's
Conjecture implies the (SMC).  The latter is now immediate from
Theorem~\ref{theo:extbonin} and the former establishes the equivalence
of the two conjectures.
\begin{corollary}
  Kantor's Conjecture holds true if and only if the Sticky Matroid
  Conjecture holds true.
\end{corollary}

\section{Proof of Theorem~\ref{theo:extbonin}}\label{sec:extbonin}
We start with a proposition that states that the so called Escher matroid (\cite{Oxley} Fig.\ 1.9) is not a matroid.
For easier readability we use lattice theoretic notation here, i.e. $x \vee y$ for $\cl(x \cup y)$,
$x \wedge y$ for $\cl(x \cap y)$ and $x \le y$ for $cl(x) \subseteq cl(y)$.

\begin{prop}\label{prop:escher}
  Let  $l_1,l_2,l_3$ be three lines in a matroid  that are pairwise
  coplanar but not all lying in a plane. If $l_1$ and $l_2$ intersect in a
  point $p$, then $p$ must also be contained in $l_3$.
\end{prop}
\begin{proof}
  By submodularity of the rank function we have
  \[r\left((l_1\vee l_3) \wedge (l_2\vee l_3)\right) \le \rk(l_1\vee
  l_3) + (l_2\vee l_3) - \rk(l_1\vee l_2 \vee l_3 )=
  3+3-4=2.\] Now $l_3 \vee p \le (l_1\vee l_3) \wedge (l_2\vee l_3)$
  and hence $p$ must lie on $l_3$.
\end{proof}

Probably the easiest way to prove that the (SMC) holds for rank 3 is
to proceed as follows. If a rank-3 matroid $M$ is not modular, then it has a
pair of disjoint lines. We consider two extensions $N_1$ and $N_2$  of $M$ such that
$N_1$ adds to the two lines a point of intersection and $N_2$ erects a
V\'amos-cube ($V_8$ in \cite{Oxley}) using the disjoint lines as base
points. By Proposition~\ref{prop:escher} the amalgam of $N_1$ and
$N_2$ cannot exist (see Figure~\ref{fig:poljakturzik}).

\begin{figure}
  \centering
 \includegraphics[width=.6\textwidth]{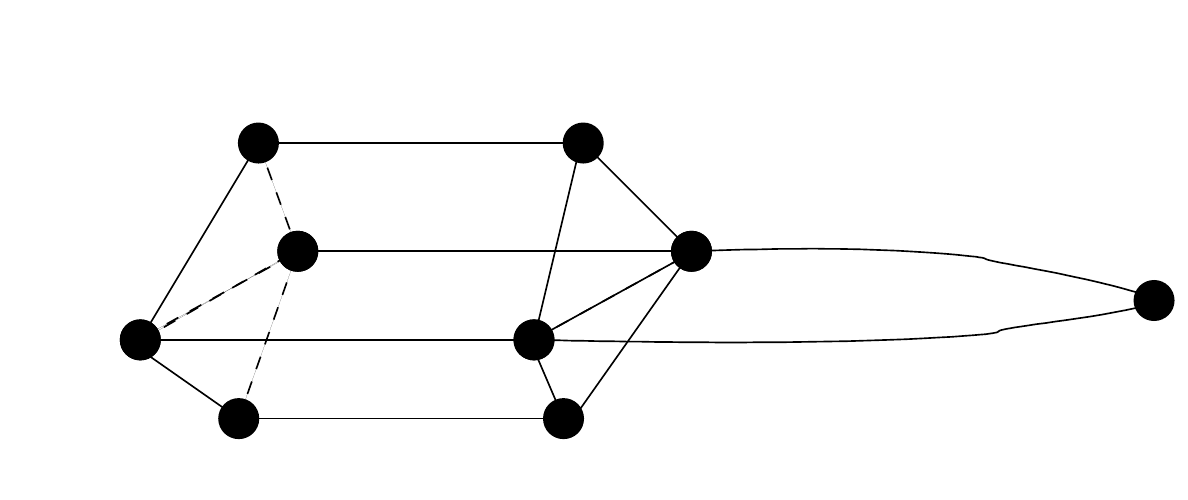}  
  \caption{This is not a matroid \label{fig:poljakturzik}}
\end{figure}

Bonin~\cite{BoninSticky} generalized this idea to the situation of a
disjoint line-hyperplane pair in {matroids of} arbitrary {rank}. We further
generalize this to a non-modular pair of a hyperplane $H$ and a flat
$F$  that can be made modular by a proper extension.  Our first aim is
to show that such a pair exists in any matroid  that is not OTE.
Again, the following is immediate:

\begin{prop}\label{modularePaareBleibenInErweiterungenGleich}
  Let $ M $ be a matroid, $ M' $ an extension of $ M $ and $ (X,Y)
  $ a modular pair of flats in $ M $. Then $
  (\cl_{M'}(X), \cl_{M'}(Y)) $ is a modular pair in $ M' $.  Moreover
\[\cl_{M'}(X) \cap \cl_{M'}(Y) = \cl_{M'}(X \cap Y).\]
\end{prop}

\begin{prop}\label{prop:intersectableModularCut}
  Let $M$ be a matroid, $\mathcal{M}$ a modular cut in $M$ and
  $M' = M+_{\mathcal{M}}p$ the corresponding single-element extension. If $M'$ does not
  contain a modular pair of flats $X'=X\cup p, Y'=Y \cup p$
  such that $X$ and $Y$ are a non-modular pair in $M$, then
\[\mathcal{M}':=\{\cl_{M'}(F) \mid F \in \mathcal{M}\}\]
is a modular cut {in $M'$}.
\end{prop}

\begin{lemma}\label{lemma:smallestModularDefect}
  Let $M_0$ be a matroid  that is not OTE and $(X,Y)$ be a non-modular
  pair of smallest modular defect $\delta := \delta(X,Y)$ such that
  there is a single-element extension decreasing their modular defect. 
  Then there exists a sequence $M_1,\ldots,M_\delta$ of matroids such
  that $M_{i}$ is a single-element extension of $M_{i-1}$ for $i=1,\ldots,
  \delta$ and
  $\delta_{M_i}\left(\cl_{M_i}(X),\cl_{M_i}(Y)\right)=\delta -i$. In
  particular $(\cl_{M_\delta}(X),\cl_{M_\delta}(Y))$ are a modular
  pair in $M_\delta$.
\end{lemma}
\begin{proof}
  Let $\mathcal{M}$ denote the modular cut generated by $X$ and $Y$ in $M_0$.
  Inductively we conclude, that by the choice of $X$ and $Y$
  \[\mathcal{M}_i:=\{\cl_{M_i}(F) \mid F \in \mathcal{M}\}\] is a
  modular cut in $M_i$ for $i=1,\ldots,\delta-1$ implying the
  assertion.
\end{proof}

\begin{lemma}\label{lem:minplusmax}
  Let $M$ be a matroid that is not OTE.  Then there exists an
  intersectable non-modular pair $(F,H)$ of smallest modular defect,
  where $F$ is a minimal element in the modular cut
  $\mathcal{M}_{F,H}$ generated by $H$ and $F$, and $H$ is a
  hyperplane of $M$.
\end{lemma}

\begin{proof}
  Since $M$ is not OTE, it is not modular and hence of rank at least
  three.  Every non-modular pair of flats in a
    rank-3 matroid clearly satisfies the assertion. Hence we may assume 
    $\rk(M) \geq 4$.  Let $(X,Y)$ be a non-modular intersectable
  pair of flats in $M$ of smallest modular defect $\delta_{min}$ and
  chosen such that, {first}, $X$ is of minimal and, {second}, $Y$ of maximal
  rank. We claim that $F=X$ and $H=Y$ are as required. Let
  $\mathcal{M}_{X,Y}$ be the modular cut generated by these two flats.

   Assume, contrary to the first assertion, that there exists an $ F \in
   \mathcal{M}_{X,Y}$ with { $ F \subsetneq X $.
   Since the principal modular cut $\mathcal{M}_{X \cap Y}$ contains $X$ and $Y$, it is a superset
   of the modular cut $\mathcal{M}_{X,Y}$. Hence we obtain $X \cap Y \subseteq F$.
   Since $\mathcal{M}_{X,Y}$ contains $F$ and $Y$ but not $X \cap Y = F \cap Y$}, the pair $(F,Y)$ is non-modular and intersectable in $M$
   (according to Proposition~\ref{prop:intersectableModularCut}).  Due
   to submodularity of $\rk$ we have $\rk(X) + \rk(F \cup Y) \geq
   \rk(X \cup Y) + \rk(F) $ and hence:
\begin{align*}
\delta(F,Y) &= \rk(F) + \rk(Y) - \rk(F \cup Y) - \rk(F \cap Y)  \\
&\leq \rk(X) + \rk(Y) - \rk(X \cup Y) - \rk(X \cap Y) = \delta(X,Y) = \delta_{min},
\end{align*}
contradicting the choice of $X$.  Next we show that
  $\cl(X \cup Y) = E(M)$. Assume to the contrary that there exists $p
  \in E(M) \setminus \cl(X \cup Y) $ and let $Y_1 = \cl(Y \cup
p)$.  Then $X \cap Y = X \cap Y_1$ and hence $\delta(X,Y_1) =
\delta(X,Y)$. Since $\mathcal{M}_{X,Y_1} \subseteq \mathcal{M}_{X,Y}$,
the pair $(X,Y_1)$ remains intersectable, contradicting the choice of
$Y$, and hence verifying   $\cl(X \cup Y) = E(M)$. 
Finally, assume $Y$ is not a hyperplane.  Let $ Y' = \cl(Y \cup p) $
with $p \in X \setminus Y$.  Then
\begin{align*}
\delta(X,Y') &= \rk(X) + \rk(Y') - \rk(X \cup Y') - \rk(X \cap Y') \\
& =\rk(X) + \rk(Y) + 1 - \rk(X \cup Y) - \rk(X \cap Y) - 1 = \delta(X,Y). 
\end{align*}
Since $Y$ is not a hyperplane and $\cl(X \cup Y) = E(M)$, we must have
$X \cap Y' \subsetneq X$, and $X$ being minimal in $\mathcal{M}_{X,Y}$
implies $X \cap Y' \notin \mathcal{M}_{X,Y}$. Now $\mathcal{M}_{X,Y'}
\subseteq \mathcal{M}_{X,Y}$ yields that $X \cap Y' \not \in
\mathcal{M}_{X,Y'}$ and thus by Proposition
\ref{prop:intersectableModularCut} the pair $(X,Y')$ is intersectable
with $\delta(X,Y') = \delta(X,Y) = \delta_{min}$, contradicting the
choice of $Y$.
\end{proof}

Lemmas \ref{lemma:smallestModularDefect} and \ref{lem:minplusmax}
now imply the following:

\begin{theorem}\label{theorem:nonOTEmatroidHyperplaneExists}
  Let $M$ be a matroid  that is not OTE.  Then there
  exist
  \begin{enumerate}[(i)]
  \item a non-modular pair $(F,H)$ where $H$ is a hyperplane of $M$
  and 
\item an extension $N$ of $M$ such that $(\cl_N(F),\cl_N(H))$ is a
  modular pair in $N$.
  \end{enumerate}
\end{theorem}

On the other hand we also have:

\begin{theorem}\label{theorem:BoninGeneral}
  Let $M$ be a matroid and $(F,H)$ a non-modular pair of disjoint
  flats, where $H$ is a hyperplane of $M$. Then there exists an
  extension $N$ of $M$ such that for every extension $N'$ of $N$,
  $(\cl_{N'}(F),\cl_{N'}(H))$ is not a modular pair in $N'$.
\end{theorem}

\begin{proof}
  We follow the idea from~\cite{BachemKernSticky} and Bonin's proof
  \cite{BoninSticky} and  erect a
  V\'amos-type matroid above $F$ and $H$.
  Clearly, $r := \rk_M(M) \ge 3$ and $ 2 \leq \rk_M(F) \leq r - 1 $.  We extend
  $M$ by first adding a set $A$ of $ r - 1 - \rk_M(F)$ elements freely
  to $H$. Next, we add, first,  a coloop $e$, and then an element $ f $
  freely to the resulting matroid, yielding an extension $N_0$ with
  groundset $E(M) \cup A \cup \{e,f\} $ and of rank $r+1$.
  Note, that $\cl_{N_0}(H) = H \cup A$. We consider the following
  sets:
  \begin{itemize}
  \item $ T_1 = F \cup A \cup e $
  \item $ T_2 = H \cup A \cup e $ 
  \item $ B_1 = F \cup A \cup f $ 
  \item $ B_2 = H \cup A \cup f$
  \end{itemize}
  Note that $(T_1,T_2)$, $(B_1,B_2)$ are non-modular pairs of
  hyperplanes of rank $r$ in $N_0$ with the same modular defect
  \[\delta(T_1,T_2)= 2r -(r+1) -(r-\rk_M(F))=
  \rk_M(F)-1=\delta(B_1,B_2).\] Any non-modular pair of hyperplanes in
  a matroid is intersectable because the modular cut generated by the
  two hyperplanes contains additionally only the groundset of the
  matroid and hence is non-principal (see Proposition
  \ref{intersectablePairs}).  In the corresponding single-element
  extension the modular defect of the hyperplane-pair decreases by
  one. If this defect is still non-zero these two hyperplanes remain
  intersectable.  Repeating this process until they
  become a modular pair, the modular defect of other
  hyperplane-pairs stays unaffected in these extensions.  This way, we
  obtain an extension $N$ of the matroid $N_0$ of rank $r+1$ with
  groundset $E(N_0) \cup P \cup Q$ where $P$ and $Q$ are independent
  sets of size $\rk_M(F) - 1$ such that $(\cl_N(T_1),\cl_N(T_2))$ and
  $(\cl_N(B_1),\cl_N(B_2))$ are modular pairs in $N$ and $P \subseteq
  \cl_N(T_1) \cap \cl_N(T_2)$ resp. $Q \subseteq \cl_N(B_1) \cap
  \cl_N(B_2)$.
We will show now that the matroid $N$
is as required.

Assume to the contrary that there exists  an extension $N'$ of $N$
such that 
 $(\cl_{N'}(F)$, $\cl_{N'}(H))$ is a modular pair.
 As $A \subseteq \cl_{N'}(H)$ and $A \cap \cl_{N'}(F) = \emptyset$ we compute
\begin{align}\label{eq:aligned1} 
&\rk_{N'}((\cl_{N'}(F) \cap \cl_{N'}(H)) \cup A) = \rk_{N'}(\cl_{N'}(F) \cap \cl_{N'}(H))) + |A| \nonumber\\
=&  \rk_{N'}(\cl_{N'}(F)) + |A| +  \rk_{N'}(\cl_{N'}(H)) - \rk_{N'}(\cl_{N'}(F) \cup \cl_{N'}(H))) \nonumber\\
= & \rk_{N'}(\cl_{N'}(F \cup A))  + \rk_{N'}(\cl_{N'}(H)) -  \rk_{N'}(\cl_{N'}(F \cup H)) \nonumber\\
=&  (r - 1) + (r - 1)  - r = r - 2.
\end{align}   
%
%
Let $D_1= \cl_{N'}(A \cup P\cup e) $ and $D_2= \cl_{N'}(A \cup Q \cup f) $. Proposition \ref{modularePaareBleibenInErweiterungenGleich} yields $\cl_{N'}(\cl_N(T_1)) \cap \cl_{N'}(\cl_N(T_2)) = \cl_{N'}(\cl_N(T_1) \cap \cl_N(T_2))$
  and it holds $\rk_{N'}(D_1) = \rk_{N'}(D_2) = r-1$. We obtain
 \begin{align}\label{eq:aligned2} 
& \rk_{N'}((\cl_{N'}(F) \cap \cl_{N'}(H)) \cup D_1) \le \rk_{N'}((\cl_{N'}(F \cup D_1) \cap \cl_{N'}((H \cup D_1)) \nonumber \\
 =& \rk_{N'}(\cl_{N'}(T_1) \cap \cl_{N'}(T_2)) = \rk_{N'}(\cl_N(T_1) \cap \cl_N(T_2)) = r - 1 = \rk_{N'}(D_1).
  \end{align}
 This implies $\cl_{N'}(F) \cap \cl_{N'}(H) \subseteq D_1$.
 Similarly, using $B_1$ and $B_2$ instead of $T_1$ and $T_2$, we get $\cl_{N'}(F) \cap \cl_{N'}(H) \subseteq D_2$ 
  and conclude $(\cl_{N'}(F) \cap \cl_{N'}(H))  \cup A \subseteq D_1 \cap D_2$. 
 This yields 
 \begin{equation}
   \label{eq:aligned3}
    \rk_{N'}(D_1 \cap D_2) \ge \rk_{N'}((\cl_{N'}(F) \cap \cl_{N'}(H))  \cup A) \stackrel{\eqref{eq:aligned1}}{=} r-2. 
 \end{equation}
From
 $\rk_{N'}(D_1 \cup D_2) = \rk_{N'}(A \cup P \cup Q \cup e \cup f)) = r+1$
  we finally obtain
 \[ \rk_{N'}(D_1) + \rk_{N'}(D_2) \stackrel{\eqref{eq:aligned2}}{=} 2r - 2 < (r+1) + (r-2) \\\stackrel{\eqref{eq:aligned3}}{\le} \rk_{N'}(D_1 \cup D_2) + \rk_{N'}(D_1 \cap D_2) \]
 contradicting submodularity. \qedhere

\end{proof}
Summarizing the two previous theorems yields the final result of this
section:
\begin{theorem}\label{theo:nine}
Let $ M $ be a matroid  that is not OTE.
Then $ M $ is not sticky.
\end{theorem}

\begin{proof}
  By Theorem \ref{theorem:nonOTEmatroidHyperplaneExists}, $ M $ has a
  non-modular intersectable pair of flats $ (F,H) $ such that $ H $ is
  a hyperplane, and there exists an extension $N_1$ of $M$ such that
  $(\cl_{N_1}(F), \cl_{N_1}(H))$ is a modular pair.  Possibly
  contracting $(F \cap H)$, and referring to Lemma 7 of
  \cite{BachemKernSticky}, we may assume that $F$ and $H$ are
  disjoint.  Thus, by Theorem~\ref{theorem:BoninGeneral}, there also
  exists an extension $N_2$ of $M$ such that in every extension $N$ of
  $N_2$ the pair $(\cl_{N}(F), \cl_{N}(H))$ is not modular.  Hence $M$
  is not sticky.
\end{proof}

\section{Hypermodularity and OTE matroids}

We collect some facts about hypermodular matroids and OTE matroids
that we need for the proof of Theorem~\ref{theo:OTErank4sticky} and
the embedding theorems in the next section. Recall that a matroid is
hypermodular if any pair of hyperplanes intersects in a coline.
Modular matroids are hypermodular and hypermodular matroids of rank at
most $3$ must be modular. Thus, a contraction of a hypermodular
matroid of rank $ n $ by a flat of rank $ n-3 $ is a modular matroid
of rank $3$.  Every projective geometry $ P(n,q) $ is hypermodular and
remains hypermodular if we delete up to $ q - 3 $ of its points.  In
the following we will focus on the case of hypermodular matroids of
rank~$4$.

\begin{prop}\label{pro:HypModMatrEbenenLinien}
  Let $ M $ be a hypermodular rank-4 matroid. If $ M $ contains a
  disjoint line and hyperplane, then $ M $ also contains two
  disjoint coplanar lines. {The same holds for a modular cut in $M$.}
\end{prop}

\begin{proof}  
  Let $ (l_1,e_1) $ be a disjoint line-plane pair in $ M $. Take a
  point $ p$ in $e_1 $. Because of hypermodularity, the plane $ l_1
  \lor p $ intersects the plane $ e_1 $ in a line $ l_2 $ in $ M $.
  The lines $ l_1 $ and $ l_2 $ are coplanar and disjoint.  If now
  $l_1$ and $e_1$ are elements of a modular cut $\mathcal{M}$ in $M$
  then it holds also $l_2 \in \mathcal{M}$.
\end{proof}

The next results are matroidal versions of similar results of Klaus
Metsch {(see \cite{METSCH1989161})} for linear spaces.

\begin{lemma}\label{lemma:1}
  Let $M$ be a hypermodular matroid of rank $4$ on a groundset $E$.
  Let $l_1, l_2$ be two disjoint coplanar lines. Then $E$ can be
  partitioned into $l_1,l_2$ and lines that are coplanar with $l_1$
  and with $l_2$.
  The modular cut $\mathcal{M}$ generated by $l_1$ and
    $l_2$ always contains such a line-partition of $E$.
\end{lemma}

\begin{proof}
  We set $e=\cl(l_1 \cup l_2)$. Then $l_p:=(l_1 \vee p) \wedge (l_2
  \vee p)$ is a line for every $p \in E \setminus e$ and coplanar to
  $l_1$ and $l_2$. By Proposition~\ref{prop:escher} it must be
  disjoint from $l_1$ and $l_2$ {and from $e$}.  This together with
  Proposition~\ref{prop:escher} implies that for $q \in E \setminus e
  $ with $p \ne q$ we must have either $l_p \wedge l_q= 0$ or
  $l_p=l_q= p \vee q$.  We denote the set of lines constructed this
  way by $\Delta$.  Now we choose a line $l_{p^*} \in \Delta$ and for
  each $r \in e \setminus (l_1 \cup l_2)$ we get a line $l_r = e
  \wedge (l_{p^*} \vee r)$.  Let $\Sigma$ be the set of lines obtained
  in that way. It is clear that $\Sigma$ is a line partition of $e
  \setminus (l_1 \cup l_2)$.  Again, Proposition~\ref{prop:escher}
  implies that these lines must be pairwise disjoint and disjoint
  from $l_1,l_2, l_{p^*}$ and all lines $l_q \in \Delta$. Now, the set
  $\Gamma = l_1 \cup l_2 \cup \Sigma \cup \Delta$ is the desired set
  of lines partitioning $E$. Obviously, it holds $\Gamma \subseteq
  \mathcal{M}$. \qedhere
  
\end{proof}

A non-trivial and non-principal modular cut in a
  matroid always contains a non-modular pair of flats.
Proposition~\ref{pro:HypModMatrEbenenLinien} implies, that in a
  hypermodular rank-4 matroid it even must contain two disjoint
  coplanar lines. By Lemma~\ref{lemma:1} we, thus, get a set of
  pairwise disjoint lines that partition the ground set.
Moreover we have:
\begin{theorem}\label{theorem:2}
  \begin{enumerate}
  \item Under the assumptions of Lemma~\ref{lemma:1} the following two statements are equivalent: 
      \begin{enumerate}[(a)] 
      \item There exists a single-element extension $M'$ where $\cl_{M'}(l_1)$ and $\cl_{M'}(l_2)$ intersect. 
      \item The modular cut generated by $l_1$ and $l_2$ in $M$
        contains a set of pairwise coplanar lines,
        $l_1$ and $l_2$ among them, partitioning the
        groundset $E(M)$.
      \end{enumerate}
  \item If a single-element extension $M'$ as in (i) exists, then the
    restriction to $M$ of any  line in $M'$ is a line.
  \item If there is no single-element extension as in (i), the matroid
    $M$ contains two non-coplanar lines $l_3,l_4$ such that $l_i$ and
    $l_j$ are coplanar for all $i \in \{1,2\}$ and $j \in \{3,4\}$ and
    no three of them are coplanar, i.e.,\ it has the V\'amos matroid
    containing $l_1$ and $l_2$ as a restriction.
     \end{enumerate}
\end{theorem}

\begin{proof}
  (i) By Lemma \ref{lemma:1} the modular cut $\mathcal{M}$ generated
  by $l_1$ and $l_2$ contains a set of lines partitioning the
  groundset $E$.  Since any two of these lines intersect in the
  extension $M'$ in the new point, they must be coplanar.

  On the other hand, if we have a set $\Gamma$ of pairwise coplanar
  lines partitioning the groundset $E$, $l_1$ and $l_2$ among them,
  these lines must form the minimal elements of a modular cut. This is
  seen as follows.  Consider the set $\mathcal{M}$ of flats in $M$
  which are elements or supersets of elements of $\Gamma$.  Any two
  lines of $\mathcal{M}$ are disjoint and coplanar, hence they do not
  form a modular pair.  For $p\in E$ let $l_p$ denote the line in
  $\Gamma$ containing $p$ and let $h \in \mathcal{M}$ be a hyperplane
  containing $p$. Then $h$ contains $l_p$ or some other line $l$ that
  is coplanar with $l_p$.  Since in the second case $l_p \le l \vee p
  \le h$ we always have $l_p \le h$.  Let $h_1 \ne h_2$ be two
  hyperplanes in $\mathcal{M}$, let $l=h_1 \wedge h_2$ and $p \ne q$
  be two points on $l$.  Then $l_p \le h_i$ and $l_q \le h_i$ for $i
  \in \{1,2\}$ implying $l_p=l_q=l$.  Finally, consider a hyperplane
  $h$ and a line $l$. If they are a modular pair then they must
  intersect in a point $r$, hence $l=l_r$ and $l \le h$.
    Thus $ \mathcal{M} $ is a modular cut
    defining a single-element extension where $l_1$ and $l_2$ intersect.
  
(ii) Let $p$ denote the new point and $l$ a line containing $p$.
    Let $q$ be another point on $l$. {Then $q$ is contained in a line $l_q$ in $M$ of the 
    partition of $E(M)$ in lines. In $M'$ we obtain $\{p,q\} \subseteq \cl_{M'}(l_q)$. 
    Since $\{p,q\} \subseteq l$ we obtain $\cl_{M'}(l_q) = l$ hence 
     the restriction of $l$ to
    $M$ is the line $l_q$.}

  (iii) Let $\Gamma = l_1 \cup l_2 \cup \Sigma \cup \Delta$ be the
    line-partition of the groundset $E$ from the proof of Lemma
    \ref{lemma:1}. By (i) there exist $l_3$ and $l_4$ in
 {$\Gamma \setminus \{l_1,l_2\}$ that} are not
  coplanar and hence $l_3 \cup l_4 \not\subseteq \cl(l_1
  \cup l_2) = e$. {If $l_3,l_4 \in \Delta$ we are done
    hence we may assume} that $l_3=l_r \in \Sigma$ and $l_4=l_q \in
  \Delta$ where $l_q = (l_1 \vee q) \wedge (l_2 \vee q)$ and $l_r = e
  \wedge (l_{p^*} \vee r)$, as in the proof of Lemma \ref{lemma:1}.
  Since $l_{p^\ast}$ and $l_3$ are coplanar we conclude $l_{p^\ast}
  \ne l_4$. If $l_{p^\ast}$ and $l_4$ are not coplanar, we replace
  $l_3$ by $l_{p^\ast}$ and are done. Hence we may assume that they
  are coplanar.  The hyperplanes $l_4 \vee r$ and $ l_{p^\ast} \vee r$
  intersect in the line $l_3'=(l_4 \vee r) \wedge (l_{p^\ast} \vee
  r)$.  Assuming $l_3' \le e$ yields $l_3'=(l_4 \vee r) \wedge
  (l_{p^\ast} \vee r)\wedge(l_1 \vee l_2)=l_r=l_3$, contradicting
  $l_3$ and $l_4$ being not coplanar. Hence $l_3'$ intersects $e$ only
  in $r$.  Furthermore, by Proposition~\ref{prop:escher}, $l_3'$ must
  be disjoint from $l_{p^\ast}$ and $l_4$. Choose $p'$ on $l_3'$ but
  not on $e$ and define $l_3'':=l_{p'} \in \Delta$. We claim that
  $l_{p'}$ must be noncoplanar with at least one of $l_{p^\ast}$ or
  $l_4$. Otherwise, we would have
\[l_3''=(l_{p^\ast} \vee l_{p'}) \wedge (l_{4} \vee l_{p'})=
(l_{p^\ast} \vee p') \wedge (l_{4} \vee p')=(l_{p^\ast} \vee l_3')
\wedge (l_{4} \vee l_3')=l_3'\] which is impossible since $l_3''\in
\Delta$ is disjoint from $e$.   \qedhere
\end{proof}

The absence of a configuration in Theorem~\ref{theorem:2} (iii) is
called bundle condition in the literature.

\begin{definition}\label{def:BundleCondition}
  A matroid $ M $ of rank at least $4$ satisfies the {\em bundle
    condition} if for any four disjoint lines $ l_1, l_2,l_3, l_4 $ of $
  M $, no three of them coplanar, the following holds: If five of
  the six pairs $ (l_i,l_j) $ are coplanar, then all pairs are
  coplanar.
\end{definition}

Since a non-modular pair of hyperplanes together with the entire groundset
always forms a modular cut that is not principal, OTE matroids must be hypermodular. 
Hence, Theorem~\ref{theorem:2} has the following corollary:

\begin{corollary}\label{lem:Rang4HFMBCIstModular}
Let $ M $ be an OTE matroid of rank $4$. If the bundle-condition in $ M $ holds, then $ M $ is modular.
\end{corollary}
\begin {proof} Let $M$ be a rank-4 OTE-matroid that is not modular.
  Then, because $M$ is hypermodular and because of Proposition
  \ref{pro:HypModMatrEbenenLinien} it contains two disjoint coplanar
  lines.  From Theorem~\ref{theorem:2} (iii) follows that the
  bundle-condition does not hold in $M$.  
  \end{proof}

\section{Embedding Theorems}\label{sec:embed}

With these results, we can prove a first embedding theorem. Assertion
(iii) is a result of Kahn~\cite{Kahn1980}.

\begin{theorem}\label{satz:ErsterEinbettungssatz}
  Let $ M $ be a hypermodular rank-4 matroid with a finite or countably
  infinite groundset. Then $ M $ is embeddable in an OTE matroid $
  \overline{M}$ of rank 4 where the restriction of any line of $\overline{M}$ is a
  line in $ M $. Furthermore:
\begin{enumerate}
\item $ \overline{M} $ is finite if and only if $ M $ is finite.
\item The simplification of $\overline{M}/p $ is isomorphic to the
  simplification of $ M/p$ for every $ p \in E(M) $.
\item If $ M $ fulfills the bundle-condition then $ \overline{M} $ is modular.
\end{enumerate}
\end{theorem}

\begin{proof} Let $P$ be a list of all disjoint coplanar pairs of
  lines of $M$. Clearly, $P$ is finite or countably infinite. We
  inductively define a chain of {matroids $ M = M_0, M_1, M_2 \hdots
    $} as follows: Let $ M_0 = M $, suppose $ M_{i-1} $ has already
  been defined for an $ i \in \mathbb{N} $.
  Let $ l_{i1} $ and $ l_{i2} $ denote the pair of disjoint lines in
  the list at index $i$. If $ l_{i1} $ and $ l_{i2} $ are not
  intersectable in the matroid $ M_{i-1} $, set $ M_i = M_{i-1} $.
  Otherwise, let $ M_i $ be the single-element extension of $ M_{i-1}
  $ corresponding to the modular cut $\mathcal{M}_{i-1}$ generated by
  $ l_{i1} $ and $ l_{i2} $ in $ M_{i-1} $.

  By Theorem~\ref{theorem:2} (ii), the restriction of a line in $
  M_{i+1} $ is a line in $ M_i $ and hence is also a line in $ M $.
  As a consequence also the restriction of a plane in $M_{i+1} $ is a
  plane in $ M $ hence two planes in $M_{i+1}$ intersect in a line.
  Thus all matroids $ M_i $ are hypermodular of rank $4$.  Now let $
  \overline{M} $ be the set system $
  (\overline{E},\overline{\mathcal{I}}) $ where $
  \overline{\mathcal{I}} \subseteq \mathcal{P}(\overline{E}) $, $
  \overline{E} = \bigcup_{i=0}^{\infty}(E(M_i)) $ and $ I \in
  \overline{\mathcal{I}} $ if and only if $ I $ is independent in some
  $ M_i $.  Clearly, $\overline{\mathcal{I}}$ satisfies the
  independence axioms of matroid theory.  We call $ \overline{M} $ the
  \textit{union of the chain of matroids}.  {The matroid
  }$\overline{M} $ is hypermodular of rank $4$ and has no new lines as
  well.

  Assume there were a modular cut $ \overline{\mathcal{M}} $ in $
  \overline{M} $ that is not principal.  By
  Proposition~\ref{pro:HypModMatrEbenenLinien} it contains a pair of
  disjoint coplanar lines. The restriction of this pair in $ M $ is on
  the list, say with index $ i $.  The modular cut $\mathcal{M}_{i-1}$
  generated by these two lines in $ M_{i-1} $ must contain $
  \cl_{M_{i-1}}(\emptyset) $, otherwise the lines would intersect in $
  M_i $, hence also in $ \overline{M} $.  Since $\{
  \cl_{\overline{M}}(X) | X \in \mathcal{M}_i\} \subseteq
  \overline{\mathcal{M}}$ we also must have $
  \cl_{\overline{M}}(\emptyset) \in \overline{\mathcal{M}}$, a
  contradiction to $\overline{\mathcal{M}}$ not being principal. Thus,
  $\overline{M}$ is OTE.  If $M$ is finite, so is the list $P$ and
  hence $\overline{M}$ proving (i).

It suffices to show that {for every point $p \in E(M)$ every point $q
  \in E(\overline{M}) - (E(M)\cup p)$} is parallel to a point in
$M/p$. As the restriction of the line spanned by $p$ and $q$ in
$\overline{M}$ is a line in $M$ it contains a point different from $p$
and (ii) follows. Finally, (iii) is
Corollary~\ref{lem:Rang4HFMBCIstModular}. \qedhere
\end{proof}

This embedding theorem has the following corollary:
\begin{corollary}\label{corollar:KantorRang4}
  Kantor's conjecture is reducable to the rank-4 case.
\end{corollary}
\begin{proof}
  Assume Kantor's conjecture holds for rank-4 matroids.  Let $M$ be a
  finite hypermodular matroid of rank $n>4$.  All contractions of $M$
  by a flat of rank $n-4$ are finite hypermodular matroids of rank 4,
  hence are embeddable into a modular matroid.  Using Theorem
  \ref{satz:ErsterEinbettungssatz}, it is easy to see that these
  contractions are also \textit{strongly embeddable} (as defined in
  \cite{Kantor}, Definition 2) into a modular matroid. Hence the
  matroid $M$ satisfies the assumptions of Theorem 2 in \cite{Kantor},
  and thus is embedabble into a modular matroid, implying the general
  case of Kantor's Conjecture.
\end{proof}

Similarly, it is easy to show that our Conjecture \ref{conj:1} is
reducible to the rank-4 case.  We have a second embedding theorem:

\begin{theorem}\label{satz:ZweiterEinbettungssatz}
  Let $M$ be a matroid of finite rank on a set $E$ where $E$ is finite
  or countably infinite. Then $M$ is embeddable in an OTE matroid of
  the same rank.
\end{theorem}

\begin{proof} We proceed similar to the proof of
  Theorem~\ref{satz:ErsterEinbettungssatz}. Let $ P $ be the list of
  all intersectable non-modular pairs of $ M $.  We build a chain {of
    matroids $ M = M_0, M_1 \hdots $, }where each matroid $M_{i+1} $
  is the extension of $M_i$, where the modular defect of the $i$-th
  pair on the list can no longer be decreased.  Let $ \overline{M} $
  be the union of the extension chain as in the proof before. Then $
  \overline{M} $ is a matroid of finite rank with a finite or
  countably infinite ground set.  If there still are
  {intersectable} non-modular pairs in $ \overline{M} $
  we repeat the process and obtain $ \overline{M} _1$.  This yields a
  chain { of matroids $ \overline{M} , \overline{M} _1,
    \overline{M} _2, \hdots $.  Let $ \overline{\overline{M}} $ }be
  the union of that extension chain.  Clearly, $
  \overline{\overline{M}} $ is a matroid. We claim it is OTE. For
  assume it had a non-trivial modular cut generated by a
  {non-modular} pair of intersectable flats $f_1,f_2$.
  Since their rank is finite, there exists an index $k$ such that the
  matroid $ \overline{M} _k $ contains a basis of $f_1$ as well as of
 $f_2$. But then in the matroid $\overline{M} _{k+1}$ the pair would
  not be intersectable anymore and we get a contradiction. Thus, $
  \overline{\overline{M}} $ is an OTE matroid.
\end{proof}

We have a similar result for hypermodular matroids:

\begin{theorem}\label{satz:KorollarZweiterEinbettungssatz}
Every matroid $M$ of finite rank $r$ with {finite or }countably infinite groundset is embeddable in a infinite hypermodular matroid $\overline{\overline{M}} $ of rank $r$.
\end{theorem}

\begin{proof}
  The proof mimics the one of
  Theorem~\ref{satz:ZweiterEinbettungssatz}, except that we have only
  the non-modular pairs of hyperplanes in the list.  This generalizes
  the technique of \textit{free closure} of rank-3 matroids and it is
  not difficult to show (see e.g.\ Kantor~\cite{Kantor}, Example 5)
  that if $M$ is non-modular (hence $\rk \geq 3$),  {every} contraction
  of $\overline{\overline{M}} $ by a flat of rank $\rk-3$ in $\overline{\overline{M}} $  {is an} infinite projective
  non-Desarguesian plane and hence $\overline{\overline{M}} $ must be infinite, too.
\end{proof}

\section{On the Non-Existence of Certain Modular Pairs in Extensions
  of OTE Matroids}

In order to prove that the proper amalgam exists for any two
extensions of a finite rank-4 OTE matroid we need some technical
lemmas. We will show that certain modular pairs cannot exist in
extensions of rank-4 OTE matroids.  We need some preparations for
that.

\begin{prop}\label{HochstaettlerLemma}
Let $ M $ be matroid with groundset $ T $, let $ (X,Y) $ be a modular pair of subsets of $ T $ and let $ Z \subseteq X \setminus Y $. 
Then $ (X \setminus Z,Y) $ is a modular pair, too.
\end{prop}

\begin{proof}
  Submodularity  implies $\rk(X \cup Y) -\rk(X) \le
  \rk((X \setminus Z) \cup Y)-\rk(X \setminus Z)$. Using modularity
  of $(X,Y)$ we find
  \begin{eqnarray*}
\rk(X \setminus Z) +\rk(Y) &=& \rk(X \cup Y) + \rk((X \setminus Z) \cap Y) -\rk(X) + \rk(X
  \setminus Z)  \\&\le& \rk((X \setminus Z)\cup Y) +\rk((X \setminus Z)
  \cap Y)    
  \end{eqnarray*}
and another application of submodularity implies the assertion.
\end{proof}

By $(D)$ we abbreviate the following list of assumptions:
\begin{itemize} 
\item $ M $ is a matroid with groundset $ T $ and rank function $ \rk $. 
\item $ M' $ is an extension of $ M $ with rank function $ \rk' $ and groundset $ E' $. 
\item $ X$ and $Y$ are subsets of $E'$ such that $ X \cap T =
  l_X$ and $Y \cap T = l_Y $ are two disjoint coplanar lines in $ M $.
\item $ X \cap Y $ is a flat in $ M' $. 
\end{itemize}

\begin{prop}\label{NotInClosureLemma}
  Assume $(D)$ and, furthermore, that $ X \setminus T \subseteq Y $
  and that $ (X,Y) $ is a modular pair of sets in $ M'$. Then
$x \notin \cl_{M'}(Y)$ for all  $x \in l_X$.
\end{prop}

\begin{proof}
  Assume to the contrary that there exists $ x\in l_X $ with $ x \in
  \cl_{M'}(Y) $. Then coplanarity of $ l_X $ and $ l_Y $
 implies
 \[X \cap T = l_X \subseteq l_X \lor l_Y = x \lor l_Y \subseteq
 \cl_{M'}(Y).\] Hence $X \subseteq \cl_{M'}(Y)$, implying $\rk'(Y) =
 \rk'(X \cup Y) $ and  modularity of $ (X,Y) $ yields
 $ \rk'(X) = \rk'(X \cap Y) $, a contradiction, because $ X \cap Y
 $ is a flat in $M'$ and a proper subset of $ X $.
\end{proof} 

\begin{lemma}\label{LinieErsetzRangGleich}
  Assume $(D)$ and that $ M $ is of rank 4 (the rank of $ M' $ may be
  larger) and, furthermore,
\begin{itemize}
\item $ (X,Y) $ is a modular pair of sets in $ M' $ with $ X \setminus T \subseteq Y $ and $ T \nsubseteq \cl_{M'}(X \cup Y) $ and
\item $ l' \subseteq T $ is a line disjoint coplanar to $ l_X $ and $ l_Y $, not lying in $ l_X \lor l_Y $.
\end{itemize}
Then $ X' = (X \setminus T) \cup l' $ implies $ \rk'(X') = \rk'(X) $.
\end{lemma}

\begin{proof}
  Choose $ x \in l_X $ and $ x' \in l' = X' \cap T $.  Because $ l_X $
  and $ l_Y $ are coplanar and $ X \setminus T \subseteq Y $ we conclude $ \cl_{M'}(x \cup Y) = \cl_{M'}(X \cup Y)$.
  Similarly, we get $  \cl_{M'}(x' \cup Y) = \cl_{M'}(X' \cup Y) $.

  By assumption $M$, being of rank $4$, is spanned by $l',l_X$ and
  $l_Y$ and hence $ T \subseteq \cl_{M'}(\{x, x'\} \cup Y) $.
  If we had $ x' \in \cl_{M'}(x \cup Y) $, then this would imply that $ T
  \subseteq \cl_{M'}(x \cup Y) = \cl_{M'}(X \cup Y) $, contradicting
  the assumptions, thus $ x' \notin \cl_{M'}(x \cup Y) $. In
  particular $ x' \notin \cl_{M'}(X). $

{ Proposition~\ref{NotInClosureLemma}  yields $ x \notin \cl_{M'}(Y) $. If we had $ x \in \cl_{M'}(x' \cup Y) $
  using the  exchange-axiom of the closure-operator we would find
  $ x' \in \cl_{M'}(x \cup Y)$ which is impossible. } Hence {we obtain }
  $ x \notin \cl_{M'}(x' \cup Y) = \cl_{M'}(X' \cup Y) $.
  In particular $ x \notin \cl_{M'}(X') $.  

  The choice of $x$ and $x'$ implies $ \cl_M(l_X \cup x') =
  \cl_{M}(l' \cup x) $ and using $X \setminus T=X' \setminus T$ we  obtain $\cl_{M'}(X \cup x') = \cl_{M'}(X' \cup x). $ We
  conclude 
\[\rk'(X') + 1 = \rk'(X' \cup x) = \rk'(X \cup x') = \rk'(X) + 1, \]
{hence $\rk'(X') = \rk'(X)$.}  
\end{proof}

\begin{lemma}\label{modPaarGehtNicht1}
  Assume $(D)$, $ M $ is a rank-4 OTE matroid and  $ X
  \setminus T \subseteq Y $, $ Y \setminus T \subseteq X $ and $ T
  \nsubseteq \cl_{M'}(X \cup Y) $.  Then 
 $(X,Y)$ is not a modular pair in $M'$.
\end{lemma}

\begin{proof}

  { OTE matroids are hypermodular, hence} $ M $ is hypermodular, OTE
  and of rank $4$. By Theorem \ref{theorem:2} (iii), it has two lines
  $ l_1 $ und $ l_2 $ that span $M$ but are both disjoint coplanar to
  $ l_X $ and $ l_Y$ and disjoint to $l_X \vee l_Y$.

Assume that  $(X,Y)$ were a modular pair in $M'$.
Let $ X' = (X \setminus T) \cup l_1 $ and $ Y' = (Y \setminus T) \cup l_2 $. Then by Lemma \ref{LinieErsetzRangGleich}
\begin{equation}\label{RaengeXX'YY'GenauGleich}
\rk'(X') = \rk'(X) \text{ and } \rk'(Y') = \rk'(Y).
\end{equation} 
Since $ T \subseteq \cl_{M}(l_1,l_2) \subseteq \cl_{M'}(X' \cup Y') $ and $
T \nsubseteq \cl_{M'}(X \cup Y) $ we get {
\begin{equation}\label{RangX'CupY'Groesser}
\rk'(X \cup Y) < \rk'(X \cup Y \cup T) = \rk'(X' \cup Y' \cup T) = \rk'(X' \cup Y').
\end{equation} }
By definition $X' \cap Y' = (X \setminus T)\cap (Y \setminus T) = X \cap Y$ and hence by sumodularity {
\begin{align*}
\rk'(X \cup Y) + \rk'(X \cap Y) &< \rk'(X' \cup Y') + \rk'(X' \cap Y') &\text{by \eqref{RangX'CupY'Groesser}}\\
&\leq \rk'(X') + \rk'(Y') \\
&= \rk'(X) + \rk'(Y) &\text{by \eqref{RaengeXX'YY'GenauGleich}}
\end{align*}}
contradicting $(X,Y)$ being a modular pair.
\end{proof}

We come to the main result of this section.

\begin{theorem}\label{ModularePaareNichtMÃ¶glich} 
  Let $ M $ be a rank-4 OTE matroid with groundset $T$ and $M'$ an
  extension of $M$ with ground set $E'$.  Let $X,Y \subseteq E'$ be
  sets such that $X \cap Y$ is a flat in $M'$ and the
  restrictions $l_X=X \cap T$ and $l_Y=Y \cap T$ are disjoint coplanar lines
  in $M$.  If $ T \nsubseteq \cl_{M'}(X \cup Y) $ then
  $(X, Y) $ is not a modular pair in $ M' $.
\end{theorem} 

\begin{proof}
  Assume to the contrary that $ (X,Y) $ were a modular pair in $ M' $.
  Let $ X' = (X \cap T) \cup (X \cap Y) $ and $ Y' = (Y \cap T) \cup
  (X \cap Y) $.  Applying Proposition \ref{HochstaettlerLemma} twice,
  we find that the pair $ (X',Y') $ is modular in $ M' $, too, and
  satisfies the assumptions of Lemma~\ref{modPaarGehtNicht1} yielding the required contradiction.
\end{proof}

By contraposition we get

\begin{corollary}\label{Cor:ModularePaareNichtMÃ¶glich}
Let $ M $ be a rank-4 OTE matroid with groundset $T$ and $M'$ an extension of $M$. 
Let $(X,Y)$ be a modular pair of flats in $M'$ such that $(X \cap T,Y \cap T)$ is a non-modular pair in $M$.
Then $ T \subseteq \cl_{M'}(X \cup Y) $.
\end{corollary} 

Regarding the case that $(X \cap T, Y \cap T)$ is a disjoint line-plane pair,
we show the following.

\begin{lemma}\label{ReduktionFall2c}
  Let $ M $ be a rank-4 OTE matroid with groundset $ T $ and rank
  function $ \rk $ and let $ M' $ be an extension of $ M $ with
  groundset $ E' $ and rank function $ \rk' $.  Assume that $ X,Y
  \subseteq E' $ are sets such that $ X \cap T = e_X $ is a plane, $Y
  \cap T = l_Y $ a line disjoint from $ e_X $ in $M$, and that $ X
  \cap Y $ is a flat in $ M' $.  Assume that there exists a line $l_X
  \subseteq e_X$ coplanar with $l_Y$ such that $\rk'((X \cap Y) \cup
  e_X) = \rk'((X \cap Y) \cup l_X) + 1$.  Then $(X,
    Y)$ is not a modular pair in $M'$.
\end{lemma}

\begin{proof}
  Assume, for a contradiction, $ (X,Y) $ were a modular pair 
  in $ M' $ and let $ X' = (X \cap Y) \cup e_X $.  Since $ X' = X
  \setminus Z $ with $ Z = X \setminus (Y \cup e_X) \subseteq X
  \setminus Y $, we find that by Proposition~\ref{HochstaettlerLemma}$,
  (X', Y) $ is a modular pair 
in $M'$, too.  Let $ X'' =  {(X \cap Y) \cup l_X }$. By assumption $ \rk'(X') = \rk'(X'') + 1 $ and $X''
  \cap T$ is a line disjoint from and coplanar to $l_Y$.  Moreover $
  X'' \cap Y = {X \cap Y}$, thus $ X'' \cap Y $ is a flat in $ M' $.
  Furthermore submodularity implies $ \rk'(X' \cup Y) \leq \rk'(X''
  \cup Y) + 1 $.  Because $(X',Y)$ is a modular pair we
  obtain:
  \begin{eqnarray*}
 \rk'(X'' \cup Y) + 1 + \rk'(X'' \cap Y) &\geq& \rk'(X' \cup Y) + \rk'(X' \cap Y) \\&=& \rk'(X') + \rk'(Y) = \rk'(X'') + 1 + \rk'(Y)     
  \end{eqnarray*}
  and again submodularity of $ r' $ implies that equality must hold
  throughout.  Hence $ (X'',Y) $ is a modular
  pair and
\begin{equation*}
\rk'(X'' \cup Y) + 1 = \rk'(X' \cup Y) = \rk'(X' \cup Y \cup T) =  \rk'(X'' \cup Y \cup T)
\end{equation*}
implying $ T \nsubseteq \cl_{M'}(X'' \cup Y) $. The pair $ (X'',Y) $ now
contradicts Theorem~\ref{ModularePaareNichtMÃ¶glich}.
\end{proof}

\section{The Proper Amalgam}\label{sec:amalgam}

We prove Theorem~\ref{theo:OTErank4sticky} by constructing the
 {\em proper amalgam} of two given extensions of a rank-4 OTE matroid.
In this section we define this amalgam and we analyse some of its
properties.  Throughout, if not mentioned otherwise, we assume the
following situation.

{Let} $ M $ {be} a matroid with groundset $ T $ and rank function $ r
$ and $ M_1 $ and $ M_2 $ be extensions of $ M $ with
groundsets $ E_1 $ resp.\ $ E_2 $ and rank functions $ r_1 $ resp.\ $
r_2 $, where $ E_1 \cap E_2 = T $ and $ E_1 \cup E_2 = E $.  All
matroids are of finite rank with finite or countably infinite ground
set.  We define two functions $ \eta: \mathcal{P}(E)
\to \mathbb{Z} $ und $ \xi: \mathcal{P}(E) \to
\mathbb{Z} $ by
\[\eta(X) = \rk_1(X \cap E_1) + \rk_2(X \cap E_2) - \rk(X \cap T)\]
\[\text{ and }\xi(X) = \min\{ \eta(Y) \colon Y \supseteq X \}. \]

The following is immediate:

\begin{prop}\label{XiRangaxiome}
The function $ \xi $ is subcardinal, finite and monotone. That is, 
\begin{align*}
&\textbf{R1}: 0 \leq \xi(X) \leq |X| \text{, for all } X \subseteq E. \\
&\textbf{R1a}: \text{For all } X \subseteq E \text{ there exist an } X' \subseteq X, |X'| < \infty \text{, such that } \xi(X) = \xi(X'). \\
&\textbf{R2}: \text{For all } X_1 \subseteq X_2 \subseteq E \text{ we have } \xi(X_1) \leq \xi(X_2). 
\end{align*}
Moreover $ \xi(X) \leq \eta(X) $ for all $ X \subseteq E $.  
\end{prop}

If $ \xi $ is submodular on $ \mathcal{P}(E) $, then $ \xi $ is the
rank function of an amalgam of $ M_1 $ and $ M_2 $ along $ M $ (see
eg.\ \cite{Oxley}, Proposition 11.4.2).  This amalgam, if it exists,
is called the \textit{proper amalgam of $ M_1 $ and $ M_2 $ along $ M
  $}.

Now let $ \mathcal{L}(M_1,M_2) $ be the set of all subsets $ X $ of $
E $, so that $ X \cap E_1 $ and $ X \cap E_2 $ are flats in $ M_1 $
resp.\ $ M_2 $.  Then it is easy to see that $ \mathcal{L}(M_1,M_2) $
with the inclusion-ordering is a complete lattice of subsets of $ E $.
Let $ \land_{\mathcal{L}} $ and $ \lor_{\mathcal{L}} $ be the meet
resp.\ the join of this lattice. Clearly, for two sets $ X,Y \in
\mathcal{L}(M_1,M_2) $ we have $ X \land_{\mathcal{L}} Y = X \cap Y $
and $ X \lor_{\mathcal{L}} Y \supseteq X \cup Y $.  We need two
results from~\cite{Oxley}. 
\begin{lemma}[see~\cite{Oxley} Prop.\ 11.4.5.]   \label{XiCharakterisierung}
For all $ X \subseteq E $ 
\[ \xi(X) = \min\{ \eta(Y) \colon Y \in \mathcal{L}(M_1,M_2) \text{ and } Y \supseteq X \}.\]
\end{lemma}

\begin{lemma}[see~\cite{Oxley} Lemma 11.4.6.]\label{kleinsteFlÃ¤che}
Let $ Y \subseteq E $ and $ Z $ be the smallest element of $ \mathcal{L}(M_1,M_2) $ containing $ Y $, then $ \eta(Z) \leq \eta(Y) $ holds.
\end{lemma}

The proof of Lemma 11.4.6 in \cite{Oxley} must be
  slightly modified in the end in order to make it work for matroids
  of finite rank but infinite groundset as well.
\begin{proof} 
  As in \cite{Oxley} for all $X \subseteq E$ we define
    $\phi_1(X) = \cl_1(X \cap E_1) \cup (X \cap E_2)$ and $\phi_2(X) =
    (X \cap E_1) \cup \cl_2 (X \cap E_2)$.  Following \cite{Oxley}
    we derive
\[\eta(\phi_i(X))\le \eta(X) \text{ for all }X \subseteq E\text{ and }i=1,2.\]

Now let $Z$ be the minimal element in $ \mathcal{L}(M_1,M_2) $ such that $Y \subseteq Z$ and choose 
$Y \subseteq W \subseteq  Z$  maximal with 
\[\eta(W)\le \eta(Y). \]
From $Y \subseteq W \subseteq \phi_i(W) \subseteq Z$ and $\eta(\phi_i(W)) \le \eta(W)$ follows $\phi_i(W) =W$ for $i=1,2$ and hence 
$W=Z \in \mathcal{L}(M_1,M_2)$ and Lemma \ref{kleinsteFlÃ¤che} follows, also implying Lemma \ref{XiCharakterisierung}.
\end{proof}

Note that the proof of this lemma and part (R1a) of Proposition \ref{XiRangaxiome} imply that Theorem~\ref{theo:modMatrSticky} holds  for infinite matroids of finite rank as well.
Now we generalize a result of Ingleton (cf.\ \cite{Oxley}, Theorem 11.4.7):

\begin{theorem}\label{EigentlichesAmalgamExistiert}
  Assume that for {any} pair $ (X,Y) $ of sets of $
  \mathcal{L}(M_1,M_2) $ the inequality defining
    submodularity is satisfied for at least one of $ \eta $ or $ \xi $.
  Then $ \xi $ is submodular on $ \mathcal{P}(E) $ and the proper
  amalgam of $ M_1 $ and $ M_2 $ along $ M $ exists.
\end{theorem} 

\begin{proof}
  Let $ X_1,X_2 \subseteq E $. By Lemma~\ref{XiCharakterisierung} we
  find $ Y_i \in \mathcal{L}(M_1,M_2) $ such that $ X_i \subseteq Y_i
  $ and $ \xi(X_i) = \eta(Y_i) $ for $ i = 1,2 $.  From $ \eta(Y_i) =
  \xi(X_i) \leq \xi(Y_i) \leq \eta(Y_i) $ we conclude that $ \xi(X_i)
  = \xi(Y_i) = \eta(Y_i) $. By assumption either $ \eta $ or $
  \xi $ or both are submodular on the pair of flats $ (Y_1,Y_2) $.
  Furthermore, $ X_1 \cap X_2 \subseteq Y_1 \cap Y_2 = Y_1
  \land_{\mathcal{L}} Y_2 $ and $ X_1 \cup X_2 \subseteq Y_1 \cup Y_2
  \subseteq Y_1 \lor_{\mathcal{L}} Y_2 $. Hence, by Proposition~\ref{XiRangaxiome} 
\[
\xi(X_1 \cap X_2) + \xi(X_1 \cup X_2) \leq \xi(Y_1
  \land_{\mathcal{L}} Y_2) + \xi(Y_1 \lor_{\mathcal{L}} Y_2).    
  \] Thus,
  if $\eta$ is submodular on $(Y_1,Y_2)$
  \begin{eqnarray*}
\xi(X_1 \cap X_2) + \xi(X_1 \cup X_2) &\leq &\eta(Y_1 \land_{\mathcal{L}} Y_2) +  \eta(Y_1 \lor_{\mathcal{L}} Y_2)\\& \leq &\eta(Y_1) + \eta(Y_2) = \xi(X_1) + \xi(X_2)    
  \end{eqnarray*}
and otherwise
\[\xi(X_1 \cap X_2) + \xi(X_1 \cup X_2)\le \xi(Y_1) + \xi(Y_2) = \xi(X_1) + \xi(X_2).\]
Hence $ \xi $ is submodular on $ \mathcal{P}(E) $ and the proper
amalgam exists.
\end{proof}

Lemma~\ref{kleinsteFlÃ¤che} immediately yields

\begin{lemma}\label{EtaVomJoinGrÃ¶sser}
If $ X,Y  $ are in $ \mathcal{L}(M_1,M_2) $, then $ \eta(X \cup Y) \geq \eta(X \lor_{\mathcal{L}} Y) $. 
Moreover we have $ \xi(X \cup Y) = \xi(X \lor_{\mathcal{L}} Y) $.
\end{lemma}

We finish this section with a small lemma. 

\begin{lemma}\label{XiGleichEtaBeiFlÃ¤chen}
Additionally to the assumptions {from the second paragraph of this section} let $ M $ be of rank $4$.
Let $ X \in \mathcal{L}(M_1,M_2) $ with $ \rk(X \cap T) \geq 2 $. Then  $ \xi(X) = \eta(X) $.
\end{lemma}

\begin{proof}
  Assume there exists $Y \supseteq X$ such that
  $\xi(X)=\eta(Y)<\eta(X)$. Then $\rk(Y \cap T)> \rk(X \cap T)$. Hence
  there exists an element $t \in (Y\cap T) \setminus X$, and because $ X
  \cap E_1, X \cap E_2 $ and $ X \cap T $ are flats we get
\begin{align*}
\eta(X \cup t) &= \rk_1((X \cup t) \cap E_1) + \rk_2((X \cup t) \cap E_2) - \rk((X \cup t) \cap T) \\
&= \rk_1(X \cap E_1) + 1 + \rk_2(X \cap E_2) + 1  - \rk(X \cap T) - 1 = \eta(X) + 1.
\end{align*}

But since $ M $ is of rank $4$ and $ \rk((X \cup t)
\cap T) \geq 3 $, the decrease of $ \eta $ for supersets of $ X \cup t $ is bounded by $1$ and thus $ \eta(Y) \geq \eta(X \cup
t) - 1 = \eta(X) $, a contradiction.
\end{proof}

\section{Proof of Theorem~\ref{theo:OTErank4sticky} }

Our proof of Theorem~\ref{theo:OTErank4sticky} may be considered as a
generalization of the proof of Proposition 11.4.9.\ in \cite{Oxley}.
Oxley refers to unpublished results of A.W.\ Ingleton.  We
start with a lemma.
\begin{lemma}\label{ZusammenfassungRang4HauptfilterEtaNichtSubmodular}  
  Let $ M $ be a rank-4 OTE matroid with ground set $ T $.  Let $ M_1
  $ and $ M_2 $ be two extensions of $ M $ with the ground sets $
  E_1,E_2 $ and rank functions $ r_1,r_2 $. Let $ E_1 \cap E_2 = T $
  and $ E_1 \cup E_2 = E $ and let $ \eta, \xi $ and $
  \mathcal{L}(M_1,M_2) $ be defined as in Section~\ref{sec:amalgam}.

  Let $ (X,Y) $ be a pair of elements of $ \mathcal{L}(M_1,M_2) $ that
  violates the submodularity of $ \eta $. Then
  \begin{enumerate}
  \item $\eta(X) + \eta(Y) - \eta(X \cap Y) - \eta(X \cup Y) \\= 
\delta(X \cap E_1, Y \cap E_1) + \delta(X \cap E_2, Y \cap E_2) - \delta(X \cap T, Y \cap T) = -1$.
\item $(X \cap E_i,Y \cap E_i)$ is a modular pair in $M_i$ for $i=1,2$.
\item $(X \cap T,Y \cap T)$ are two disjoint coplanar lines or a disjoint line-plane pair in $M$.
\item $\eta(X) = \xi(X)$  and  $\eta(Y) = \xi(Y)$. 
  \end{enumerate}
\end{lemma}
\begin{proof}
  For part (i) a straightforward computation yields the first
  equality. The second one follows from the fact that OTE-matroids are
  hypermodular and that the modular defect in a
  hypermodular rank-$4$ matroid is bounded by $1$.
  Parts (ii) and (iii) are immediate from (i) and part (iv)
  follows from Lemma~\ref{XiGleichEtaBeiFlÃ¤chen}. \qedhere
\end{proof}

\begin{lemma}\label{xiSubmodular}
  Under the assumptions of
  Lemma~\ref{ZusammenfassungRang4HauptfilterEtaNichtSubmodular}, let $
  (X,Y) $ be a pair of elements of $ \mathcal{L}(M_1,M_2) $ such that
  the submodularity of $ \eta $ in $ \mathcal{L}(M_1,M_2) $ is
  violated, and either  $ \xi(X \cup Y) < \eta(X \cup Y) $ or $
  \xi(X \cap Y) < \eta(X \cap Y) $.  Then $ \xi $ is submodular for $
  (X,Y) $ in $ \mathcal{L}(M_1,M_2) $.
\end{lemma}

\begin{proof}
  Recall that $ \xi(X \cup Y) \leq \eta(X \cup Y) $ and $ \xi(X \cap
  Y) \leq \eta(X \cap Y) $ and $ \xi(X \cap Y) = \xi(X
  \land_{\mathcal{L}} Y) $ as well as $ \xi(X \cup Y) = \xi(X
  \lor_{\mathcal{L}} Y) $ by Lemma~\ref{EtaVomJoinGrÃ¶sser}. Moreover
  by Lemma~\ref{ZusammenfassungRang4HauptfilterEtaNichtSubmodular}
  (iv), $ \eta(X) = \xi(X) $ and $ \eta(Y) = \xi(Y) $. Altogether this
  implies
\begin{align*}
&\quad \xi(X) + \xi(Y) - \xi(X \land_{\mathcal{L}} Y) - \xi(X \lor_{\mathcal{L}} Y) \\
&= \xi(X) + \xi(Y) - \xi(X \cap Y) - \xi(X \cup Y) \\
&> \eta(X) + \eta(Y) - \eta(X \cap Y) - \eta(X \cup Y) = -1
\end{align*}
proving the assertion.
\end{proof}

We are now ready to tackle the proof of
Theorem~\ref{theo:OTErank4sticky} which is an immediate consequence of
the following:

\begin{theorem} \label{theo:prAmalExists}
Let $ M $ be a rank-4 OTE matroid.
Then for any pair of extensions of $ M $ the proper amalgam exists.
\end{theorem}
\begin{proof}
  Let $T$ denote the ground set of $ M $ and $ M_1,M_2 $ be two
  extensions of $ M $ with ground sets $ E_1,E_2 $ and rank functions $
  r_1,r_2 $, such that $ E_1 \cap E_2 = T $ and $ E_1 \cup E_2 = E $.
  We show that for these two extensions the proper amalgam exists.
  Let $ \eta $ and $ \xi $ be defined as in the previous section.
  By Lemma \ref{EigentlichesAmalgamExistiert} it suffices to show that
  for each pair $ (X,Y) $ of elements of $ \mathcal{L}(M_1,M_2) $
  either $ \eta $ or $ \xi $ is submodular.

By cases, we check all possible pairs $ (X,Y) $ of sets of $
  \mathcal{L}(M_1,M_2) $ where the submodularity of $ \eta $ could be
  violated, and show that 
  $ \xi(X \cup Y) < \eta(X \cup Y) $ or $ \xi(X \cap Y) < \eta(X \cap
  Y) $ and hence (by Lemma \ref{xiSubmodular}) $ \xi $ is submodular on
  $ (X,Y) $.

By Lemma \ref{ZusammenfassungRang4HauptfilterEtaNichtSubmodular}, $ (X
\cap E_i,Y \cap E_i) $ are modular pairs of flats in $ M_i $ for $i=1,2$
and $ (X \cap T, Y \cap T) $ is a pair of disjoint coplanar lines or a
disjoint line-plane-pair.

\textbf{Disjoint coplanar lines:} Assume $ X \cap T = l_X $ and $ Y \cap
  T = l_Y $ are two disjoint coplanar lines. By Corollary \ref{Cor:ModularePaareNichtMÃ¶glich}
  the fact that $ (X \cap
  E_i,Y \cap E_i) $ are modular pairs for $i=1,2$ implies that $ T
  \subseteq \cl_{M_i}((X \cup Y) \cap E_i) $ for $i=1,2$. Let $ t \in T
  \setminus \cl_M(l_X \cup l_Y) $. Then
\begin{eqnarray*}
  &&\eta( X \cup Y \cup  t ) \\
  &=&  r_1((X \cup Y \cup  t ) \cap E_1) + r_2((X \cup Y \cup  t ) \cap E_2)  - \rk((X \cup Y \cup  t ) \cap T) \\
  &=& r_1((X \cup Y) \cap E_1) + r_2((X \cup Y) \cap E_2) - \rk((X \cup Y) \cap T) - 1 \\
  &=& \eta(X \cup Y) - 1.
\end{eqnarray*}
Hence $ \xi(X \cup Y) < \eta(X \cup Y) $.

\textbf{Disjoint point-line pair:}
Assume $ X \cap T = e_X $ is a plane and $ Y \cap T = l_Y $ is
a line disjoint from $ e_X $. By Lemma~\ref{ReduktionFall2c}
for every  line $ l \subseteq e_X $ such that
$\rk(l \vee l_Y)=3$ we must have
\begin{equation}
\label{VorausFall2b} r_i((X \cap Y \cap E_i) \cup e_X) = r_i((X \cap Y \cap E_i) \cup l) \text{ for }i=1,2.  
\end{equation}
Choose a point $ p_1 \in e_X $. Since $M$ must be hypermodular
$l_X=(e_X \wedge (l_Y \vee p_1))$ is a line in $M$ and $p_1 \in l_X$.  Since
$ Y \cap E_1 $ is a flat in $M_1$ not containing $ p_1 $ and $X \cap Y \cap E_1$ is a flat in $M_1$ disjoint from $T$ we have 
\begin{eqnarray}\label{Fall2bRYPlus1}
r_1((Y \cup p_1) \cap E_1) &=& r_1(Y \cap E_1) + 1\\
\label{Fall2bRXCapYPlus1}
r_1((X \cap Y \cap E_1) \cup p_1) &=& r_1(X \cap Y \cap E_1) + 1.
\end{eqnarray}
Choose a second point $ p_2 \in l_X $ such that $ p_2 \neq p_1 $.
Since $ l_X $ and $ l_Y $ are coplanar, we obtain 
\[p_2 \in l_X \subseteq \cl_M(p_1 \cup l_Y) = \cl_M(p_1 \cup (Y \cap
T)) \subseteq \cl_{M_1}(p_1 \cup (Y \cap E_1)) \] and thus
\begin{equation} \label{Fall2bR1YcupL1E1}
r_1((Y \cup l_X) \cap E_1) = r_1((Y \cup \{p_1,p_2\}) \cap E_1) = r_1((Y \cup p_1) \cap E_1).
\end{equation}
Furthermore, since $ \{p_1,p_2\} \subseteq l_X \subseteq X $:
\begin{equation} \label{p1p2Gleichp1}
r_1((X \cup Y \cup \{p_1,p_2\}) \cap E_1) = r_1((X \cup Y) \cap E_1) 
\end{equation}
Using these equations and the modularity of $ (X \cap E_1,Y \cap E_1) $ in $ M_1 $ we compute
\begin{eqnarray*}
&&r_1(X \cap E_1) + r_1((Y \cup \{p_1,p_2\}) \cap E_1) \\
&\overset{(\ref{Fall2bR1YcupL1E1})}{=}&  r_1(X \cap E_1) + r_1((Y \cup p_1) \cap E_1) \\
&\overset{(\ref{Fall2bRYPlus1})}{=}&  r_1(X \cap E_1) + r_1(Y \cap E_1) + 1 \\
&\overset{(Mod.)}{=} &r_1((X \cup Y) \cap E_1) + r_1(X \cap Y \cap E_1) + 1 \\
&\overset{(\ref{p1p2Gleichp1})}{=} & r_1((X \cup Y \cup \{p_1,p_2\}) \cap E_1) + r_1(X \cap Y \cap E_1) + 1 \\
&\overset{(\ref{Fall2bRXCapYPlus1})}{=}&  r_1((X \cup Y \cup \{p_1,p_2\}) \cap E_1) + r_1((X \cap Y \cap E_1) \cup p_1) \\
&\leq& r_1((X \cup Y \cup \{p_1,p_2\}) \cap E_1) + r_1((X \cap Y \cap E_1) \cup \{p_1,p_2\}) 
\end{eqnarray*}
By submodularity of $ r_1 $ the last inequality must hold with equality and  hence 
\begin{align}\label{Fall2br1lGleichl1}
&\quad r_1((X \cap Y \cap E_1) \cup l_X) = r_1((X \cap Y \cap E_1) \cup p_1).
\end{align}
By symmetry (\ref{Fall2bRXCapYPlus1}) and (\ref{Fall2br1lGleichl1})
are also valid for $ r_2 $ and $E_2$.  Recalling that $ X \cap Y \cap T =
\emptyset $, we compute
\begin{eqnarray*}
\eta((X \cap Y) \cup e_X) 
&=& {\left[ \sum_{i=1}^2 r_i((X \cap Y \cap E_i) \cup e_X) \right]} - \rk(e_X) \\
&\overset{(\ref{VorausFall2b})}{=}& {\left[ \sum_{i=1}^2 r_i((X \cap Y \cap E_i) \cup l_X) \right]} - 3 \\
&\overset{(\ref{Fall2br1lGleichl1})}{=}& {\left[ \sum_{i=1}^2 r_i((X \cap Y \cap E_i) \cup p_1) \right]} - 3 \\ 
&\overset{(\ref{Fall2bRXCapYPlus1})}{=}& {\left[ \sum_{i=1}^2 \left(r_i(X \cap Y \cap E_i) + 1\right)\right]} - \rk(X \cap Y \cap T) - 3 \\
&=& \eta(X \cap Y) - 1. 
\end{eqnarray*}
Hence $\xi(X \cap Y) < \eta(X \cap Y). $
\end{proof}

\section{Conclusion}

Now if we put the embedding theorems together with Theorem~\ref{theo:OTErank4sticky},
we get the equivalence of three conjectures:

\begin{theorem}
The following statements are equivalent:
 \begin{enumerate}
  \item All finite sticky matroids are modular. (SMC)
  \item Every finite hypermodular matroid is embeddable in a modular matroid. (Kantor's Conjecture)
 \item Every finite OTE matroid is modular.
 \end{enumerate}
\end{theorem}

\begin{proof}
  $(i) \Rightarrow (ii)$ These two statements can be reduced to the
  rank-4 case (see Theorem \ref{theo:knownForSMC} and Corollary
  \ref{corollar:KantorRang4}).  Now consider a finite hypermodular
  rank-4 matroid $M$. Because of Theorem
  \ref{satz:ErsterEinbettungssatz}, it can be embedded into a finite
  rank-4 OTE matroid $M'$ that is sticky due to Theorem
  \ref{theo:OTErank4sticky}. If $(i)$ holds then $M'$ is modular and
  $M$ can be embedded into a modular matroid and $(ii)$ holds.

  $(ii) \Rightarrow (iii)$ Let $M$ be a finite OTE matroid. It is also
  hypermodular. If $(ii)$ holds, it is embeddable into a modular
  matroid. Since $M$ is OTE, it must itself already be modular.

$(iii) \Rightarrow (i)$ Let $M$ be a finite sticky matroid. Because of Theorem \ref{theo:extbonin} it must be an OTE matroid and, if $(iii)$ holds,
must be modular and $(i)$ holds. 
\end{proof}

A slightly weaker conjecture than the (SMC) in the finite case, which
could also hold in the infinite case, is the generalization of
Theorem~\ref{theo:OTErank4sticky} to arbitrary
rank.
\begin{conjecture}
  A matroid is sticky if and only if it is an OTE matroid.
\end{conjecture}

Our proof of Theorem \ref{theo:OTErank4sticky} frequently uses the
fact that we are dealing with rank $4$ matroids.  We think there is a
way to avoid Lemma \ref{XiGleichEtaBeiFlÃ¤chen}, but the case checking
in the proof of of Theorem \ref{theo:prAmalExists} seems to become
tedious even for ranks only slightly larger than $4$.  Moreover, we
need a generalization of Theorem \ref{theorem:2} (iii) in order to
generalize Lemma~\ref{LinieErsetzRangGleich}.

\section*{Acknowledgement}
The authors are greatful to an anonymous referee who carefully read the paper, and whose comments helped to  improve its readability.

\bibliographystyle{plain} \bibliography{sticky}

\end{document}